\documentclass[11pt]{amsart}
\usepackage{amsmath,amssymb,bm}
\usepackage[english,frenchb]{babel}
\usepackage[latin1]{inputenc}
\usepackage{hyperref}

\newtheorem{theorem}{Theorem}[section]
\newtheorem{theoreme}{Th\'eor\`eme}[section]
\newtheorem{lemma}{Lemma}[section]

\numberwithin{equation}{section}
\newcommand{\R}{\mathbb{R}}

\newcommand{\Z}{\mathbb{Z}}
\newcommand{\T}{\mathbb{T}}

\begin{document}
\title[Strichartz estimates]{Strichartz estimates for the periodic non elliptic Schr\"odinger equation }

\author[N.Godet]{Nicolas Godet}
\address{CNRS \& Laboratoire de Math\'ematiques (UMR 8088), Universit\'e de Cergy-Pontoise, F-95000 Cergy-Pontoise, France.}
\email{nicolas.godet@u-cergy.fr}

\author[N. Tzvetkov]{Nikolay Tzvetkov}
\address{CNRS \& Laboratoire de Math\'ematiques (UMR 8088), Universit\'e de Cergy-Pontoise, F-95000 Cergy-Pontoise, France.}
\email{nikolay.tzvetkov@u-cergy.fr}

\maketitle
\begin{abstract}
Le but de cette note est de d\'emontrer des estimations de Strichartz optimales avec pertes de d\'eriv\'ees pour l'\'equation de Schr\"odinger non elliptique pos\'ee sur le tore de dimension $2$.
\end{abstract}

\vspace{-0.5cm}

\selectlanguage{english}

\begin{abstract}
The purpose of this note is to prove sharp Strichartz estimates with derivative losses for the non elliptic Schr\"odinger equation posed on the $2$ dimensional torus.
\end{abstract}

\vspace{1cm}
%%%%%%%%%%%%%%%%%%%%%%%%%%%%%%%%%%%%%%%%%%%
%%%%%%%%%%%%%%%%%%%%%%%%%%%%%%%%%%%%%%%%%%%

\selectlanguage{francais}
\textbf{Version fran\c caise abr\'eg\'ee}

\medskip

On consid\`ere l'\'equation de Schr\"odinger non elliptique 
\begin{equation} \label{eq0}
(i\partial_t +\partial_x^2-\partial_y^2)u=0,\quad u(0,x,y)=u_0(x,y),
\end{equation}
pos\'ee sur le tore de dimension deux $\T^2=(\R| 2 \pi \Z)^2$. La solution de cette equation de donn\'ee initiale $u_0$ est donn\'ee par $e^{-it P}u_0$ o\`u $P= -\partial_x^2+ \partial_y^2$. Un couple de nombres r\'eels $(p,q)$ est dit  admissible si 
$$
\frac{1}{p}+\frac{1}{q}=\frac{1}{2},\quad p>2.
$$
Le th\'eor\`eme suivant donne des estimations de Strichartz optimales avec pertes de d\'eriv\'ees pour l'op\'erateur $P$. 

\begin{theoreme} \label{thm0}
Soit $(p,q)$ un couple admissible. Il existe une constante $C>0$ telle que pour toute donn\'ee initiale $u_0 \in H^{\frac{1}{p}}(\T^2)$,
\begin{equation} \label{strichartz0}
\|e^{-itP}u_0\|_{L^p_{t\in [0,1]}L^q_{xy}(\T^2)}\leq C\|u_0\|_{H^{\frac{1}{p}}(\T^2)}\,.
\end{equation}
De plus, l'estimation (\ref{strichartz0}) est optimale au sens o\`u l'in\'egalit\'e 
 \begin{equation}
\|e^{-itP}u_0\|_{L^p_{t\in [0,1]}L^q_{xy}(\T^2)}\leq C\|u_0\|_{H^{s}(\T^2)}
\end{equation}
devient fausse si $s<\frac{1}{p}$. 
\end{theoreme}

L'estimation de Strichartz (\ref{strichartz0}) a \'et\'e r\'ecemment d\'emontr\'ee dans le cas $p=q=4$ (voir \cite{W}) en utilisant une analyse globale sur le tore. La preuve du Th\'eor\`eme \ref{thm0} que nous pr\'esentons est uniquement bas\'ee sur des arguments locaux. Le r\'esultat du Th\'eor\`eme \ref{thm0} peut \^ etre utilis\'e dans l'\'etude de perturbations non lin\'eaires de (\ref{eq0}). En particulier, l'analyse de \cite{BGT1}, \cite{BGT2} implique que dans le cas de perturbations cubiques, l'\'equation est bien pos\'ee dans $H^s$, $s>1/2$.
Des estimations de Strichartz avec pertes dans le cas elliptique
\begin{equation}\label{ell0}
(i\partial_t +\partial_x^2+\partial_y^2)u=0,\quad u(0,x,y)=u_0(x,y),
\end{equation}
pos\'ee sur $\T^2$, ont \'et\'e obtenues dans \cite{B}. Pour l'\'equation (\ref{ell0}), on ne conna\^it pas les estimations optimales pour tous les couples $(p,q)$ admissibles (surtout pour $p<4$) mais l'analyse de \cite{B} montre que dans le cas $p=q=4$, on a mieux que (\ref{strichartz0}), \`a savoir que $1/p=1/4$ peut \^etre remplac\'e par n'importe quel nombre strictement positif.

\selectlanguage{english}

\section{Introduction}
Consider the non elliptic Schr\"odinger equation
\begin{equation} \label{eq}
(i\partial_t +\partial_x^2-\partial_y^2)u=0,\quad u(0,x,y)=u_0(x,y),
\end{equation}
posed on the two dimensional torus $\T^2=(\R| 2 \pi \Z)^2$. 
The solution of (\ref{eq}) is given by $e^{-itP}(u_0)$, where $P= -\partial_x^2+\partial_y^2$.
We study here Strichartz estimates with losses for (\ref{eq}) and we show that the approach of \cite{BGT1} gives optimal estimates.
We call a couple $(p,q)\in\R^2$ admissible if
$$
\frac{1}{p}+\frac{1}{q}=\frac{1}{2},\quad p>2.
$$
We have the following statement.
\begin{theorem}\label{thm}
Let $(p,q)$ be an admissible couple. There exists a constant $C>0$ such that for every $u_0\in H^{\frac{1}{p}}(\T^2)$,
\begin{equation}\label{strichartz}
\|e^{-itP}u_0\|_{L^p_{t\in [0,1]}L^q_{xy}(\T^2)}\leq C\|u_0\|_{H^{\frac{1}{p}}(\T^2)}\,.
\end{equation}
Moreover (\ref{strichartz}) is sharp in the sense that the estimate 
 \begin{equation}\label{faux}
\|e^{-itP}u_0\|_{L^p_{t\in [0,1]}L^q_{xy}(\T^2)}\leq C\|u_0\|_{H^{s}(\T^2)}
\end{equation}
fails for $s<\frac{1}{p}$.
\end{theorem}
The above result in the particular case $p=q=4$ was recently obtained in \cite{W} by using a different approach using the special choice of the $L^4$ norm and global analysis on the torus. The proof of Theorem~\ref{thm} we present here relies only on local arguments.

The result of Theorem~\ref{thm} can be used in the study of nonlinear perturbations of (\ref{eq}). In particular the analysis of \cite{BGT1}, \cite{BGT2} implies the well-posedness 
in $H^s$, $s>1/2$ in the case of cubic perturbations.

Some Strichartz estimates with losses in the case of the elliptic Schr\"odinger equation
\begin{equation}\label{ell}
(i\partial_t +\partial_x^2+\partial_y^2)u=0,\quad u(0,x,y)=u_0(x,y),
\end{equation}
posed on $\T^2$, were obtained in \cite{B}. In the context of (\ref{ell}), it seems that we do not have a clear picture what are 
the optimal Strichartz estimates for all admissible couples $(p,q)$ (especially for $p<4$). The analysis in \cite{B} shows that in the particular case $p=q=4$ one has 
better than (\ref{strichartz}), namely $1/p=1/4$ can be replaced by every positive number which is almost the scale invariant estimate (the scale invariant estimate is however known to be false).   

 By adapting our proof of (\ref{strichartz}), the same Strichartz estimates as in \cite{BGT1} in dimension $\mathrm{dim}(M_1)+ \mathrm{dim}(M_2)$ may be proved for the equation
\[
 (i  \partial_t + \Delta_{M_1}- \Delta_{M_2})u=0, \qquad (x,y) \in M_1 \times M_2,
\]
where $M_1$ and $M_2$ are compact Riemannian manifolds. We however do not have a clear understanding about the optimality of the estimates in such a situation (except when $M_1=M_2$).

In \cite{SAL}, Salort proved Strichartz estimates for the operator $P$ with a loss of $1/p+ \varepsilon$ derivatives for all $\varepsilon >0$ but without addressing the question of optimality.

\section{Proof of Theorem~\ref{thm}}
%%%
\subsection{Proof of (\ref{strichartz})}
%%%%
Let $\Delta=\partial_x^2+\partial_y^2$ be the Laplace operator. In the analysis, it is of importance that $\Delta$ commutes with $P$. As in \cite{BGT1} (see \cite[Corollary~2.3]{BGT1} and \cite[second part of page~583]{BGT1}), 
by using the Littlewood-Paley square function theorem and the Minkowski inequality, in order to prove (\ref{strichartz}), it suffices to prove that for every $\varphi\in C_0^{\infty}(\R)$, there exists $C>0$ such that for every $h\in (0,1]$, every $u_0\in L^2(\T^2)$,
\begin{equation}\label{LP}
\|\varphi(h^2 \Delta)e^{-itP}u_0 \|_{L^p_{t\in [0,1]}L^q_{xy}(\T^2)}\
\leq Ch^{-\frac{1}{p}}\|u_0\|_{L^2(\T^2)}\,.
\end{equation}
Let $\psi \in C_0^{\infty}(\R)$ be such that $\psi(-x^2)\psi(-y^2)$ equals one on the support of $\varphi(-x^2-y^2)$. 
Such a function exists since for a suitable $R>1$ the support of $\varphi(-x^2-y^2)$ is contained in the square $[-R,R]\times [-R,R]$ and thus it suffices to choose $\psi$
which equals one on $[-R^2,R^2]$. Then
\begin{equation*}
\psi(-x^2)\psi(-y^2)\varphi(-x^2-y^2)=\varphi(-x^2-y^2),\quad \forall\, (x,y)\in \R^2.
\end{equation*}
and hence
$$
\psi(h^2\partial_x^2)\psi(h^2\partial_y^2)\varphi(h^2\Delta)=\varphi(h^2\Delta).
$$
Therefore using the $L^2$ boundedness of $\varphi(h^2\Delta)$, we obtain that in order to get (\ref{LP}) it suffices to prove that  
for every $\psi\in C_0^{\infty}(\R)$, there exists $C>0$ such that for every $h\in (0,1]$, every $u_0\in L^2(\T^2)$,
\begin{equation}\label{prof}
\|\psi(h^2 \partial_x^2)\psi(h^2 \partial_y^2)e^{-itP}u_0 \|_{L^p_{t\in [0,1]}L^q_{xy}(\T^2)} \leq Ch^{-\frac{1}{p}}\|u_0\|_{L^2(\T^2)}\,.
\end{equation}
Let us denote by $K(t,x,y,x',y')$ the kernel of the map $\psi(h^2 \partial_x^2)\psi(h^2 \partial_y^2)e^{-itP}$, i.e.
$$
\big(\psi(h^2 \partial_x^2)\psi(h^2 \partial_y^2)e^{-itP}u_0\big)(t,x,y) =\int_{\T^2}K(t,x,y,x'y')u_0(x',y')dx'dy' \,.
$$
Then we have that
$$
K(t,x,y,x'y')=K_1(t,x,x')K_2(t,y,y'),
$$
where $K_1(t,x,x')$ is the kernel of $\psi(h^2 \partial_x^2)e^{it\partial_x^2}$ and 
$K_2(t,y,y')$ is the kernel of $\psi(h^2 \partial_y^2)e^{-it\partial_y^2}$. 
By \cite[Lemma~2.5 and Remark~2.6]{BGT1}, applied in the 1d case, we know that there exists $\alpha>0$ such that 
$$
|K_1(t,x,x')|\leq C|t|^{-1/2},\quad |K_2(t,y,y')|\leq C|t|^{-1/2},\quad \forall\, |t|\leq \alpha h\,.
$$
Consequently
$$
|K(t,x,y,x',y')|\leq C|t|^{-1},\quad \forall\, |t|\leq \alpha h\,.
$$
Thus we obtain that there exists $C>0$ such that for every $|t|\leq \alpha h$, every $u_0\in L^1(\T^2)$,
\begin{equation}\label{semi}
\|\psi(h^2 \partial_x^2)\psi(h^2 \partial_y^2)e^{-itP}u_0 \|_{L^\infty_{xy}(\T^2)}
\leq C|t|^{-1}\|u_0\|_{L^1(\T^2)}\,.
\end{equation}
With (\ref{semi}) in hand we can complete the proof of (\ref{strichartz}) exactly as in \cite[page~583]{BGT1}.
Indeed the $T-T^\star$ argument implies that for every interval $J$ of size $|J|\leq \alpha h$ one has 
\begin{equation}\label{J}
\int_{J}\|\psi(h^2 \partial_x^2)\psi(h^2 \partial_y^2)e^{-itP}u_0\|^p_{L^q(\T^2)}dt\leq C\|u_0\|^p_{L^2(\T^2)}.
\end{equation}
Next we cover $[0,1]$ with $N$ intervals of size $\leq \alpha h$, $N\sim h^{-1}$ and using $N$ times (\ref{J}) we infer that
\begin{eqnarray*}
\int_{0}^{1}\|\psi(h^2 \partial_x^2)\psi(h^2 \partial_y^2)e^{-itP}u_0\|^p_{L^q(\T^2)}dt
& \leq & \sum_{k=1}^{N} \int_{J_k}\|\psi(h^2 \partial_x^2)\psi(h^2 \partial_y^2)e^{-itP}u_0\|^p_{L^q(\T^2)}dt
\\
& \leq & Ch^{-1}\|u_0\|^{p}_{L^2(\T^2)}
\end{eqnarray*}
This completes the proof of (\ref{strichartz}).
%%%%%%%%%%%%
\subsection{Optimality of the estimate}
Let $f\in H^s(\T)$. Then $f(x+y)\in H^s(\T^2)$ and $f(x+y)$ is a stationary solution of (\ref{eq}). Therefore if (\ref{faux}) holds then 
\begin{equation}\label{sob}
\|f\|_{L^q(\T)}\leq C\|f\|_{H^s(\T)}\,.
\end{equation}
Inequality (\ref{sob}) is the Sobolev embedding which is known to hold for $s\geq \frac{1}{2}-\frac{1}{q}$.
It is also well known that it fails for $s< \frac{1}{2}-\frac{1}{q}$ as shows the next lemma.
\begin{lemma}\label{scale}
Inequality (\ref{sob}) fails for $s< \frac{1}{2}-\frac{1}{q}$.
\end{lemma}
\begin{proof}
It suffices to test (\ref{sob}) with
$$
f(x)=\eta(\lambda x),\quad \lambda\geq 1,\quad \eta\in C^\infty_0(-1/2,1/2)\,.
$$
We can see $f$ as a $C^\infty(\T)$ function and with this choice of $f$ the left hand-side of (\ref{sob}) behaves like
$\lambda^{-\frac{1}{q}}$ for $\lambda\gg 1$ while the right hand-side behaves like $\lambda^{s-\frac{1}{2}}$. Thus if (\ref{sob}) holds then
$
\lambda ^{-\frac{1}{q}}\lesssim \lambda ^{s-\frac{1}{2}}
$
which implies $s\geq \frac{1}{2}-\frac{1}{q}$.
\end{proof}
Using Lemma~\ref{scale}, we obtain that if (\ref{faux}) holds true then one should necessarily have
$$
s\geq \frac{1}{2}-\frac{1}{q}=\frac{1}{p}
$$
which proves the optimality of (\ref{strichartz}).

\noindent \textbf{Acknowledgements.} We thank the referee for useful remarks that improved this manuscript. The authors are supported by the ERC grant Dispeq.

\end{document}